\documentclass[11pt, reqno]{amsart}
\usepackage{amssymb, amsthm, amsmath, amsfonts,mathtools}
\usepackage{array, epsfig}
\usepackage{bbm}

\usepackage{hyperref}
\usepackage[numbers,square]{natbib}
\allowdisplaybreaks

\numberwithin{equation}{section}

  \evensidemargin
\oddsidemargin
\newcommand{\stirling}[2]{\genfrac{[}{]}{0pt}{}{#1}{#2}}
\newcommand{\stirlingb}[2]{\genfrac{\{}{\}}{0pt}{}{#1}{#2}}

\newcommand{\Log}{\mathop{\mathrm{Log}}\nolimits}

\newcommand{\ii}{{\rm{i}}}

\newcommand{\mA}{{\tilde a}}
\newcommand{\mB}{{\tilde b}}
\newcommand{\mtA}{{\tilde {\mathfrak{a}}}}
\newcommand{\mtB}{{\tilde {\mathfrak{b}}}}

\newcommand{\lA}{a}
\newcommand{\lB}{b}
\newcommand{\rA}{\mathfrak{a}}
\newcommand{\rB}{\mathfrak{b}}

\newcommand{\bS}{\mathbb{S}}

\def\dint{\textup{d}}

\newcommand{\E}{\mathbb E}

\newcommand{\R}{\mathbb{R}}
\newcommand{\N}{\mathbb{N}}

\newcommand{\C}{\mathbb{C}}
\newcommand{\Z}{\mathbb{Z}}

\renewcommand{\Re}{\operatorname{Re}}

\newcommand{\pos}{\mathop{\mathrm{pos}}\nolimits}

\newcommand{\dd}{{\rm d}}
\newcommand{\eee}{{\rm e}}

\theoremstyle{plain}
\newtheorem{theorem}{Theorem}[section]

\newtheorem{corollary}[theorem]{Corollary}
\newtheorem{proposition}[theorem]{Proposition}

\theoremstyle{definition}

\theoremstyle{remark}
\newtheorem{remark}[theorem]{Remark}

\begin{document}

\author{Zakhar Kabluchko}
\address{Zakhar Kabluchko: Institut f\"ur Mathematische Stochastik,
Westf\"alische Wilhelms-Universit\"at M\"unster,
Orl\'eans-Ring 10,
48149 M\"unster, Germany}
\email{zakhar.kabluchko@uni-muenster.de}

\title[Expected Face Numbers of Random Polytopes]{On Expected Face Numbers of Random Beta and Beta' Polytopes}

%%%%%%%%%%%%%%%%%%%%%%%%%%%%%%%%%%%%%%%%%%%%%%%%%%%%%%%%%%%%%%%%%%%%%%%%%%%%%%%%%
%%
%%The handwritten results corresponding to this paper are contained in the notes from 
%15.04.2020 (in the middel of the notes) 
%and 
%04.02.2021
%%
%%%%%%%%%%%%%%%%%%%%%%%%%%%%%%%%%%%%%%%%%%%%%%%%%%%%%%%%%%%%%%%%%%%%%%%%%%%%%%%



\begin{abstract}
The random beta polytope is defined as the convex hull of $n$ independent random points with the density proportional to $(1-\|x\|^2)^\beta$ on the $d$-dimensional unit ball, where $\beta>-1$ is a parameter. Similarly, the random beta' polytope is defined as the convex hull of $n$ independent random points with the density proportional to $(1+\|x\|^2)^{-\beta}$ on $\mathbb R^d$, where $\beta>\frac d2$.  In a previous work [Angles of random simplices and face numbers of random polytopes, Adv.\ Math.\, 380 (2021), 107612], we established exact and explicit formulae for the expected $f$-vectors of these random polytopes in terms of certain definite integrals.
In the present paper, we use purely algebraic manipulations to derive several identities for these integrals which yield  alternative formulae for the expected $f$-vectors.
Similar algebraic manipulations apply to Stirling numbers and yield the following identity:
$$
\sum_{s=0}^k \stirlingb{n-s}{d-s} (d-s) \stirling{d-s}{k-s}
=
\sum_{s=0}^k (-1)^s \stirlingb{n-s}{d} \stirling{d+1}{k-s}
=
\sum_{s=0}^{d-k} (-1)^s \stirlingb{n+1}{d-s} \stirling{d-s}{k}.
$$
\end{abstract}

\maketitle
%\tableofcontents

%\section{}

\section{Random convex cones in a half-space}\label{sec:random_convex_cones_half_space}

In this paper we are interested in the exact formulae for the expected number of faces of certain random polytopes, called beta and beta' polytopes. For example, the convex hull of $n$ points drawn independently and uniformly from the $d$-dimensional unit ball is a special case of the beta polytope.  Since the statements of main results  are somewhat obscured by heavy notation, we start with a simple special case.

\subsection{Statement of the problem}
Let $U_1,\ldots,U_n$ be random vectors drawn uniformly and independently from the $d$-dimensional upper half-sphere
$$
\bS^d_+:=\{x = (x_0,\ldots,x_{d})\in \R^{d+1}: x_0\geq 0, \|x\|=1\}.
$$
Here, $\|x\|$ denotes the Euclidean norm of $x$. The polyhedral convex cone spanned by these vectors (also known as their positive hull) is denoted by
$$
C_n = \pos(U_1,\ldots,U_n) = \left\{\sum_{i=1}^n \lambda_i U_i: \lambda_1,\ldots,\lambda_n\geq 0\right\}.
$$
The study of the random cone $C_n$, together with the closely related spherical polytope $C_n\cap \bS^d_+$, has been initiated in the work of B\'ar\'any, Hug, Reitzner and Schneider in~\cite{barany_etal} and continued in~\cite{convex_hull_sphere}.
Explicit formulae for the expected values of several quantities associated with $C_n$ such as its solid angle, denoted by $\alpha(C_n)$, and the number of $k$-dimensional faces of $C_n$, denoted by $f_k(C_n)$, have been obtained in~\cite{kabluchko_poisson_zero}.
These formulae are stated in terms of two triangular arrays or numbers, denoted by $A[n,k]$ and $B\{n,k\}$. To define these numbers, we first of all put  $A[n,k]:= B\{n,k\}:=0$ if $n<k$.
The numbers $B\{n,k\}$, indexed by $n\in \N$ and $k\in\{0,\ldots,n\}$,  are defined by
\begin{align}
B\{n,k\} &:= \frac {1}{(k-1)!(n-k)!} \int_0^\pi (\sin x)^{k-1} x^{n-k} \dint x, \quad \text{ if } k\in \{1,\ldots,n\},\label{eq:def_B_n_k}\\
B\{n,0\} &:= B\{n,1\} = \frac{\pi^n}{n!}, \qquad n\in\N.
\end{align}
The numbers $A[n,k]$, indexed by $n\in\N_0$ and $k\in \Z$ with $k\leq n$,  may be defined by
$$
A[n,k] := \frac{n!}{(n-k)!} \frac 1 {\pi} \int_{-\infty}^{+\infty} (\cosh x)^{-n-1} \left(\frac \pi 2 + \ii x\right)^{n-k} \dd x;
$$
see~\cite[Section~6.9 and Eqn.~(6.8)]{kabluchko_formula} for this as well as~\cite[Theorem~1.1]{kabluchko_poisson_zero} and~\cite[Section~2.2]{kabluchko_formula} for other, equivalent definitions of $A[n,k]$.
For our purposes, only the following relations satisfied by $A[n,k]$ and $B\{n,k\}$ will be important:
\begin{align}
A[n+2,k] - A[n,k] &= (n+1)^2 A[n,k-2] \qquad \text{ for all $n\in\N_0$ and $k\in\Z$,} \label{eq:A_rec_citation}\\
B\{n,k\} - B\{n,k+2\} &= (k+1)^2 B\{n+2,k+2\} \qquad \text{ for all $n\in\N$ and $k\in \N_0$.} \label{eq:B_rec_citation}
\end{align}
The proofs of these relations can be found in~\cite[Lemmas~3.8 and~3.9]{kabluchko_poisson_zero}.
The formulae for the expected number of $k$-faces of the random cone $C_n$ can now be stated as follows.
%is an array of numbers defined as follows. Consider a sequence of polynomials given by $Q_0(x) = Q_1(x) = 1$ and, for $n\in\{2,3,\ldots\}$,
%\begin{align}
%Q_n(x)
%%&= \prod_{\substack{j\in \{1,\ldots,n-1\} \\ j \not\equiv n \Mod{2}}} (1+ j^2 x^2) \label{eq:def_Q_n}\\
%%&=
%=
%\begin{cases}
%(1+ (n-1)^2x^2) (1+(n-3)^2 x^2)\ldots (1 + 3^2x^2) (1 + 1^2 x^2),&\text{ if $n$ is even,}\\
%(1+ (n-1)^2x^2) (1+(n-3)^2 x^2)\ldots (1+ 4^2 x^2) (1+ 2^2 x^2),&\text{ if $n$ is odd}.
%\end{cases}
%\label{eq:def_Q_n}
%\end{align}
%% where the last factor in the finite product is $1 + x^2$ or $1 + 2^2x^2$ depending on whether $n$ is even or odd.
%Then,
%\begin{equation}\label{eq:def_A_n_k}
%A[n,k] =
%\begin{cases}
%[x^k]  Q_n(x), &\text{ if $k$ is even},\\
%[x^k] \left(\tanh \left(\frac{\pi}{2x}\right)  \cdot  Q_n(x)\right), &\text{ if $k$ is odd and $n$ is even},\\
%[x^k] \left(\cotanh \left(\frac{\pi}{2x}\right) \cdot Q_n(x)\right), &\text{ if $k$ is odd and $n$ is odd}.
%\end{cases}
%\Big((1+ (n-1)^2x^2) (1+(n-3)^2 x^2) (1+(n-5)^2 x^2) \ldots\Big).
%\end{equation}
%Let $\mathcal Z_d$ be the $d$-dimensional Poisson zero polytope with $d\in\N$. Then, for all $\ell\in \{0,\ldots,d\}$ we have
%\begin{equation}\label{eq:theo:main}
%\E f_\ell(\mathcal Z_d) = \frac{\pi^{d-\ell}}{(d-\ell)!} A[d,d-l],
%\end{equation}
%We are now going to state some applications of Theorem~\ref{theo:main} to a natural class of random spherical polytopes.
%The next theorem summarizes the results on the expected number of $$
%In fact, we can even compute the complete expected $f$-vector of the spherical polytope  for every finite $n$.
\begin{theorem}\label{theo:spherical_polytope_f_vector}
For all $d\in\N$, $n\geq d+1$, and all $k\in \{1,\ldots,d\}$, we have
\begin{align}
\E f_{k} (C_n)
&=
\frac{n!\pi^{k-n}}{k!} \sum_{\substack{s=0,2,4,\ldots \\ s\leq d-k}} B\{n, d-s\}(d-s-1)^2 A[d-s-2,k-2]\label{eq:theo:spherical_polytope_f_vector1}\\
&=
\frac{n!\pi^{k-n}}{k!} \sum_{\substack{s=0,2,4,\ldots \\ s\leq d-k}} B\{n+s, d\} A[d,k+s].\label{eq:theo:spherical_polytope_f_vector2}
\end{align}
Also, the following complementary formulae hold:
\begin{align}
\binom{n}{k}  - \E f_{k} (C_n)
&=
\frac{n!\pi^{k-n}}{k!} \sum_{\substack{s=2,4,6,\ldots \\ s\leq n-d}} B\{n,d+s\}(d+s-1)^2 A[d+s-2,k-2]\label{eq:theo:spherical_polytope_f_vector1compl}\\
&=
\frac{n!\pi^{k-n}}{k!} \sum_{\substack{s=2,4,6,\ldots \\ s\leq n-d}} B\{n-s,d\} A[d,k-s].\label{eq:theo:spherical_polytope_f_vector2compl}
\end{align}
%where $A[n,k]$ and $B\{n,k\}$ are as above.
If  $k=1$ and $d$ is odd, then the right-hand sides of~\eqref{eq:theo:spherical_polytope_f_vector1} and~\eqref{eq:theo:spherical_polytope_f_vector1compl} involve the term $0^2 A[-1,-1]$ which we interpret as $A[1,1] - A[-1,1] = A[1,1] = 2/\pi$ in accordance with~\eqref{eq:A_rec_citation}.
\end{theorem}
Let us make some comments on this theorem. Equation~\eqref{eq:theo:spherical_polytope_f_vector1} has been obtained in~\cite[Theorem~2.2]{kabluchko_poisson_zero}. The equivalence of~\eqref{eq:theo:spherical_polytope_f_vector1} and~\eqref{eq:theo:spherical_polytope_f_vector2} constitutes the main result of this section and will be proved by purely algebraic manipulations in Theorem~\ref{theo:identity_A_B_1}. To interpret the complementary formulae~\eqref{eq:theo:spherical_polytope_f_vector1compl} and~\eqref{eq:theo:spherical_polytope_f_vector2compl},  we observe that with probability $1$, every $k$-face of $C_n$ is spanned by some $k$ distinct vectors from the set $\{X_1,\ldots,X_n\}$. Hence,  the quantity $\binom nk - f_k(C_n)$ counts the number of $k$-tuples of vectors that did not become faces.  Equation~\eqref{eq:theo:spherical_polytope_f_vector1compl} follows from~\eqref{eq:theo:spherical_polytope_f_vector1} together with the following identity, proved in~\cite[Lemma~4.3]{kabluchko_poisson_zero}: For all $n\in\N$ and all $k\in \{1,\ldots,n-1\}$ (but not for $k=n$) we have
\begin{align}
\sum_{\substack{j\in \{k,\ldots,n\} \\ j \text{ is even}}} B\{n,j\}  (j-1)^2 A[j-2,k-2] &= \sum_{\substack{j\in \{k,\ldots,n\} \\ j \text{ is odd}}} B\{n,j\}  (j-1)^2 A[j-2,k-2]= \frac{\pi^{n-k}}{(n-k)!}.\label{eq:identity_2A}
%\sum_{\substack{j\in \{k,\ldots,n\} \\ j \text{ even}}} B\{n,j\}  (j-1)^2 A[j-1,k-2] &= \frac{\pi^{n-k}}{(n-k)!}. \label{eq:identity_2B}
\end{align}
Observe that~\eqref{eq:identity_2A} just states that~\eqref{eq:theo:spherical_polytope_f_vector1} remains true for $n=d$ and $n=d-1$ provided we exclude the cases $k=d$ and $k\in \{d,d-1\}$, respectively.  Finally, the equivalence of~\eqref{eq:theo:spherical_polytope_f_vector1compl} and~\eqref{eq:theo:spherical_polytope_f_vector2compl} will be established in Theorem~\ref{theo:identity_A_B_1}.

Next we shall state formulae for the solid angle of the cone $C_n$ which is defined by
$
\alpha(C_n)
:=
%\frac{\sigma_{d} (C_n\cap \bS_+^d)}{\omega_{d+1}}
\sigma_{d} (C_n\cap \bS_+^d)/\sigma_{d}(\bS^d)
$,
where $\sigma_{d}$ denotes the $d$-dimensional surface measure on the sphere $\bS^d$. The expected solid angle is linked to the number of $1$-dimensional faces of $C_n$ by the following spherical Efron-type identity, see~\cite[Equation~(26)]{barany_etal} (and also~\cite[Theorem 2.7]{convex_hull_sphere} for a generalization):
$$
\E \alpha(C_n) = \frac 12 \left(1 - \frac {\E f_1(C_{n+1})} {n+1}\right).
$$
Combining this identity with Theorem~\ref{theo:spherical_polytope_f_vector} and Identity~\eqref{eq:identity_2A}, we arrive at the following formulae the first of which has already been obtained in~\cite[Theorem~2.5]{kabluchko_poisson_zero}.
\begin{theorem}\label{theo:expected_angle}
For all $d\in\N$ and $n\geq d+1$ we have
\begin{align}
\E \alpha(C_n)
&=
\frac{n!}{2\pi^{n}} \sum_{\substack{s=2,4,6,\ldots \\ s\leq n+1-d}} B\{n+1, d+s\} (d+s-1)^2  A[d+s-2,-1]\\
&=
\frac{n!}{2\pi^{n}} \sum_{\substack{s=2,4,6,\ldots\\s\leq n+1-d}} B\{n+1-s,d\}A[d, 1-s].
%\frac{n!}{2\pi^n} \sum_{\substack{m=1,2\ldots\\ d+2m \leq n+1}} B\{n+1, d + 2m\} (A[d+2m,1] - A[d+2m-2,1]).
\end{align}
Also, the following complementary formulae hold:
\begin{align}
\frac 12 - \E \alpha(C_n)
&=
\frac{n!}{2\pi^{n}} \sum_{\substack{s=0,2,4,\ldots \\ s\leq d-1}} B\{n+1, d-s\} (d-s-1)^2  A[d-s-2,-1]\\
&=
\frac{n!}{2\pi^{n}} \sum_{\substack{s=0,2,4,\ldots\\s\leq d-1}} B\{n+1+s,d\} A[d, 1+s].
\end{align}
%with the convention $A[-1,1] =A[0,1]=0$.
\end{theorem}

\subsection{Identities for \texorpdfstring{$A[n,k]$}{A[n,k]} and \texorpdfstring{$B\{n,k\}$}{B\{n,k\}}}
The main result of this section are the following identities proving that~\eqref{eq:theo:spherical_polytope_f_vector1} is equivalent to~\eqref{eq:theo:spherical_polytope_f_vector2} and, similarly, \eqref{eq:theo:spherical_polytope_f_vector1compl} is equivalent to~\eqref{eq:theo:spherical_polytope_f_vector2compl}.
\begin{theorem}\label{theo:identity_A_B_1}
For all $n\in \N$, $d\in \{1,\ldots,n\}$ and $k\in \{1,\ldots,d+1\}$ we have
\begin{align}
\sum_{\substack{s=0,2,4,\ldots\\s\leq d-k}} B\{n,d-s\}(d-s-1)^2 A[d-s-2,k-2]
&=
\sum_{\substack{s=0,2,4,\ldots\\s\leq d-k}} B\{n+s,d\}A[d, k+s],\label{eq:identity_A_B_1}\\
\sum_{\substack{s=2,4,6,\ldots\\s\leq n-d}} B\{n,d+s\}(d+s-1)^2 A[d+s-2,k-2]
&=
\sum_{\substack{s=2,4,6,\ldots\\s\leq n-d}} B\{n-s,d\}A[d, k-s],\label{eq:identity_A_B_2}
\end{align}
where we recall the convention $0^2 A[-1,-1] := A[1,1] - A[-1,1] = A[1,1] = 2/\pi$.
\end{theorem}
\begin{proof}
Using the recurrence relations~\eqref{eq:A_rec_citation} and~\eqref{eq:B_rec_citation},  we shall transform the sum on the left-hand side of~\eqref{eq:identity_A_B_1} in a way which can be thought of  as a ``chemical reaction'' in which the $A$-terms disintegrate, then the terms regroup, and finally the $B$-terms rejoin. The result is a sum which differs from the original one by a shift of the indices $n,k$ plus an additional term, the ``sediment'' of the reaction. If we let the terms in the sum react over and over again, more and more sediment terms are going to appear. At some point, the sum becomes empty and the reaction stops.  At this point we conclude that the right-hand side of~\eqref{eq:identity_A_B_1} equals the sum of the sediment terms.
Let us be more specific. We denote the sum on the left-hand side of~\eqref{eq:identity_A_B_1} by
$$
S_d^-(n,k) := \sum_{\substack{s=0,2,4,\ldots\\s\leq d-k}} B\{n,d-s\}(d-s-1)^2 A[d-s-2,k-2].
$$
The following considerations hold for arbitrary $n\in \N$, $d\in \N$ and $k\in \N$.  We start by applying the recurrence relation~\eqref{eq:A_rec_citation} satisfied by the numbers  $A[\cdot,\cdot]$ to rewrite our sum as follows:
$$
S_d^-(n,k) = \sum_{\substack{s=0,2,4,\ldots\\s\leq d-k}} B\{n,d-s\}(A[d-s,k]- A[d-s-2,k]).
$$
Now we extract the first term  as follows:
$$
S_d^-(n,k) = B\{n,d\}A[d,k] + \sum_{\substack{s=2,4,6,\ldots\\s\leq d-k}} B\{n,d-s\} A[d-s,k] - \sum_{\substack{s=0,2,4,\ldots\\s\leq d-k-2}} B\{n,d-s\}  A[d-s-2,k].
$$
Note that in the second sum the term with $s= d-k$ or $s=d-k-1$ vanishes which is why we dropped  it. Now we shift the index in the first sum and take both sums together:
$$
S_d^-(n,k) = B\{n,d\}A[d,k] + \sum_{\substack{s=0,2,4,\ldots\\s\leq d-k-2}}(B\{n,d-s-2\} - B\{n,d-s\})  A[d-s-2,k].
$$
Applying the recurrence relation~\eqref{eq:B_rec_citation} satisfied by the numbers $B\{\cdot,\cdot\}$, we arrive at
$$
S_d^-(n,k) = B\{n,d\}A[d,k] + \sum_{\substack{s=0,2,4,\ldots\\s\leq d-k-2}} B\{n+2,d-s\} (d-s-1)^2 A[d-s-2,k].
$$
On the right-hand side we recognize the same sum we started with, but with $(n,k)$ replaced by $(n+2,k+2)$:
$$
S_d^-(n,k) = B\{n,d\}A[d,k] + S_{d}^-(n+2,k+2).
$$
Iterating this identity $r$ times, we obtain
\begin{align*}
S_d^-(n,k) &= B\{n,d\}A[d,k] + B\{n+2,d\}A[d,k+2] + \ldots + B\{n+2r-2,d\}A[d,k+2r-2]
\\
&+  S_{d}^-(n+2r,k+2r).
\end{align*}
%For a suitable $r$, we either have $k+2r=d+1$ (in which case $S_{d}^-(n+2r,k+2r)=0$ because the sum is empty) or $k+2r = d$ (in which case $S_{d}^-(n+2r,k+2r)=B\{n-k+d,d\}A\{d,d\}$).
For a sufficiently large $r$ we have $k+2r>d$ and therefore $S_{d}^-(n+2r,k+2r)=0$ because the sum becomes empty.
The proof of~\eqref{eq:identity_A_B_1} is complete.

The proof of~\eqref{eq:identity_A_B_2} is similar, but now the ``reaction'' begins with the disintegration of the $B$-terms. First of all, \eqref{eq:identity_A_B_2} is trivially fulfilled for $n=1,2$ since in this case both sums become empty. In the following we assume that $n\geq 3$, $d\in \N$ and $k\in \N$.    Applying the recurrence relation for the $B$-numbers given in~\eqref{eq:B_rec_citation}, regrouping the terms, and then applying the recurrence relation for the $A$-numbers given in~\eqref{eq:A_rec_citation} yields
\begin{align*}
S_d^+(n,k)
&:=
\sum_{\substack{s=2,4,6,\ldots\\s\leq n-d}} B\{n,d+s\}(d+s-1)^2 A[d+s-2,k-2]\\
&=
\sum_{\substack{s=2,4,6,\ldots\\s\leq n-d}} (B\{n-2,d+s-2\} - B\{n-2,d+s\}) A[d+s-2,k-2]\\
&=
B\{n-2,d\}A[d,k-2] + \sum_{\substack{s=4,6,8,\ldots\\s\leq n-d}} B\{n-2,d+s-2\} A[d+s-2,k-2]\\
&
\phantom{=B\{n-2,d\}A[d,k-2]}
-
\sum_{\substack{s=2,4,6,\ldots\\s\leq n-d-2}} B\{n-2,d+s\} A[d+s-2,k-2]\\
&=
B\{n-2,d\}A[d,k-2] + \sum_{\substack{s=2,4,6,\ldots\\s\leq n-d-2}} B\{n-2,d+s\} A[d+s,k-2]\\
&\phantom{=B\{n-2,d\}A[d,k-2]}
-
\sum_{\substack{s=2,4,6,\ldots\\s\leq n-d-2}} B\{n-2,d+s\} A[d+s-2,k-2]\\
&=
B\{n-2,d\}A[d,k-2] + \sum_{\substack{s=2,4,6,\ldots\\s\leq n-d-2}} B\{n-2,d+s\} (A[d+s,k-2] - A[d+s-2,k-2])\\
&=
B\{n-2,d\}A[d,k-2] + \sum_{\substack{s=2,4,6,\ldots\\s\leq n-d-2}} B\{n-2,d+s\} (d+s-1)^2 A[d+s-2,k-4]\\
&=
B\{n-2,d\}A[d,k-2] + S_d^+(n-2,k-2).
\end{align*}
Iterating this identity $r$ times, we arrive at
\begin{align*}
S_d^+(n,k)
&=
B\{n-2,d\}A[d,k-2] + B\{n-4,d\}A[d,k-4] + \ldots + B\{n-2r,d\}A[d,k-2r]\\
& +  S_d^+(n-2r,k-2r).
\end{align*}
For a suitable $r$, we have $1\leq n-2r\leq d+1$,  the last sum becomes empty and the proof of~\eqref{eq:identity_A_B_2} is complete.
\end{proof}
%Repetition: The sums appearing on the left-hand side in Theorem~\ref{theo:identity_A_B_1} can be thought of as partial sums corresponding to the following full sum which we evaluated in Lemma~4.3 of~\cite{kabluchko_poisson_zero}:
%\begin{align}
%\sum_{\substack{j\in \{k+1,\ldots,n\} \\ j \text{ is even}}} B\{n,j\}  (j-1)^2 A[j-2,k-1] &= \sum_{\substack{j\in \{k+1,\ldots,n\} \\ j \text{ is %odd}}} B\{n,j\}  (j-1)^2 A[j-1,k-1]= \frac{\pi^{n-k-1}}{(n-k-1)!},\label{eq:identity_2A}
%%\sum_{\substack{j\in \{k,\ldots,n\} \\ j \text{ even}}} B\{n,j\}  (j-1)^2 A[j-1,k-2] &= \frac{\pi^{n-k}}{(n-k)!}. \label{eq:identity_2B}
%\end{align}
%for all $n\in\N$ and $k\in \{0,1,\ldots,n-2\}$ (but not for $k=n-1$).

Combining Theorem~\ref{theo:identity_A_B_1} with~\eqref{eq:identity_2A}, we obtain the following

\begin{corollary}
For all $d\in \N$ and $\ell\in \N$ we have
\begin{equation}\label{eq:sum_A_B_joint_second_argument}
\sum_{\substack{r\in \{d,\ldots,d+\ell\}\\ r \text{ is even}}} B\{r,d\}A[d,r-\ell]
=
\sum_{\substack{r\in \{d,\ldots,d+\ell\}\\ r \text{ is odd}}} B\{r,d\}A[d,r-\ell]
=
\frac{\pi^\ell}{\ell!}.
\end{equation}
\end{corollary}
\begin{remark}
The identity fails to hold for $\ell=0$ because then we have $B\{d,d\}A[d,d] = 2$.
\end{remark}
\begin{proof}
Put $n:= d+\ell +1$ and $k:=d+1\leq n-1$. Note that $\ell = n-k$.  Taking the sum of the both identities in Theorem~\ref{theo:identity_A_B_1}, and using $j$ as a substitute for $d+s$ or $d-s$, we obtain
\begin{align*}
\sum_{\substack{j\in \{k,\ldots,n\}\\ j-d \text{ is even}}} B\{n,j\} (j-1)^2 A[j-2,k-2]
&=
\sum_{\substack{s\in \{d-n,\ldots,d-k\}\\ s \text{ is even}}} B\{n+s,d\} A[d,k+s]\\
&=
\sum_{\substack{r\in \{d,\ldots,d+\ell\}\\ d+\ell-r \text{ is odd}}} B\{r,d\} A[d,r-\ell],
\end{align*}
where in the second line we defined the new summation index $r:= n+s$. By~\eqref{eq:identity_2A}, the left-hand side of this identity equals $\pi^\ell/\ell!$, which proves one of the identities in~\eqref{eq:sum_A_B_joint_second_argument} depending on the parity of $d+\ell$. To prove the remaining identity, we take $n:=d+\ell$ and $k:= d\leq n-1$. We again have $\ell = n-k$ and can repeat the same argument as above with the only difference that the condition that $s= r-n$ is even will be translated into the requirement that $d+ \ell -r$ is even.
\end{proof}

\section{Beta polytopes}

Let us now define the random polytopes we are interested in. In the present section, we state our results on beta polytopes (which are contained in the unit ball), while the next section will be devoted to beta' polytopes.

\subsection{Statement of the result}
Let $X_1,X_2,\ldots$ be independent and identically distributed (i.i.d.) random points in the $d$-dimensional unit ball with the following probability density function:
\begin{equation}\label{eq:def_f_beta}
f_{d,\beta}(x)= \frac{ \Gamma\left( \frac{d}{2} + \beta + 1 \right) }{ \pi^{ \frac{d}{2} } \Gamma\left( \beta+1 \right) } \left( 1-\left\| x \right\|^2 \right)^\beta, \qquad x\in\R^d, \;\; \|x\| <  1,
\end{equation}
where $\|x\| = (x_1^2+\ldots+x_d^2)^{1/2}$ is the Euclidean norm of the vector $x= (x_1,\ldots,x_d)\in\R^d$, and $\beta$ is a real parameter satisfying  $\beta>-1$. This so-called $d$-dimensional \textit{beta distribution} has been introduced in the works of Miles~\cite{miles} and Ruben and Miles~\cite{ruben_miles}. For example, the special case $\beta=0$ corresponds to the uniform distribution on the unit ball $\{\|x\|\leq 1\}$, whereas the weak limit of the beta distribution as $\beta\downarrow -1$ is the uniform distribution on the unit sphere $\{\|x\|=1\}$; see~\cite[Proof of Corollary~3.9]{beta_polytopes_temesvari}.  We always include the latter limit case into our results by allowing the value $\beta=-1$.

The convex hull $P_{n,d}^\beta :=[X_1,\ldots,X_n]$ is called the \textit{beta polytope} with parameter $\beta\geq -1$. The expected number of $k$-dimensional faces of the beta polytope has been determined exactly and explicitly in~\cite[Section~7.1]{kabluchko_formula}, which strongly relied on the results of~\cite{beta_polytopes}; see also~\cite{buchta_mueller,buchta_mueller_tichy,affentranger,affentranger_exact,AS92,beta_polytopes_temesvari,kabluchko_poisson_zero,kabluchko_angles,kabluchko_algorithm} for earlier results in this direction.
In order to state the corresponding formula we need to introduce some notation. Fix some  $\alpha > 0$ and consider the following analytic function:
\begin{equation}\label{eq:def_F}
F(x) = \int_{-\pi/2}^x (\cos y)^{\alpha} \dd y, \qquad x\in \C\backslash \left(\left(-\infty, -\frac \pi 2\right)\cup \left(+\frac \pi2,\infty\right)\right) =: \mathcal H.
\end{equation}
The integration is performed along any contour connecting $-\pi/2$ and $x$ and contained in the doubly slit complex plane $\mathcal H$. By Cauchy's theorem, the integral does not depend on the choice of the contour.  We choose the branch of $(\cos x)^\alpha = \eee^{\alpha \Log \cos x}$ that is equal to $1$ at $x=0$. Since $\cos x$ does not vanish in the doubly slit complex plane $\mathcal H$, the branch is well defined.
For $\nu,\kappa\in \C$ such that $\nu-\kappa\in \N_0$ define, following~\cite[Sections~5.1, 5.2]{kabluchko_formula},
\begin{align}
\rA[\nu,\kappa]
&:=
\frac{\alpha^{\nu - \kappa + 1}}{(\nu - \kappa)!} \cdot \frac {\Gamma(\nu+1)} {2\pi \ii} \int_{- \ii \infty}^{+\ii \infty} (\cos x)^{-\alpha \nu} (F(x))^{\nu-\kappa} \dd x, \qquad
\Re \kappa > 0,\label{eq:def_A_0}\\
\rB\{\nu, \kappa\}
&:=
\frac{\alpha^{\nu-\kappa}}{\Gamma(\kappa)(\nu-\kappa)!}\int_{-\pi/2}^{+\pi/2} (\cos x)^{\alpha \kappa}
(F(x))^{\nu-\kappa} \dd x,
\qquad
\Re \kappa > -\frac 1 \alpha.
\label{eq:def_B}
\end{align}
For $\nu, \kappa\in \C$ such that $\nu - \kappa \in \{-1,-2,\ldots\}$ we put $\rA[\nu,\kappa]:= \rB\{\nu, \kappa\} := 0$.  It will be shown elsewhere that for every fixed $\nu - \kappa \in \N_0$, the functions $\rB\{\nu,\kappa\}$ and $\rA[\nu,\kappa]$ can be extended to meromorphic functions of the complex variable $\nu$ (or $\kappa$). %In this se viewed as equalities of meromorphic functions.
Note that~\cite{kabluchko_formula} also uses the notation $\lB\{\nu,\kappa\} = \rB\{\nu,\kappa\}\Gamma(\kappa)$ and
$\lA[\nu,\kappa] = \rA[\nu,\kappa]/\Gamma(\nu+1)$. We are now able to state our formulae for the expected face numbers of beta polytopes.

\begin{theorem}\label{theo:beta_poly_f_vector}
Let $d\geq 3$ and $n\in\N$ be such that $n\geq d+1$. Also, let $\beta\geq -1$ and define $\alpha:= 2\beta+d$. Then, for all $k\in \{1,\ldots,d\}$, the expected number of $(k-1)$-dimensional faces of $P_{n,d}^\beta$ is given by
\begin{align}
\lefteqn{\E f_{k-1} (P_{n,d}^\beta)}\notag\\
&=
\frac{2\cdot n!}{k!} \left(\frac{\Gamma(\frac{\alpha}{2})}{2\sqrt \pi \, \Gamma(\frac{\alpha+1}{2})}\right)^{n-k}
\sum_{\substack{s=0,2,4,\ldots \\ m:=d-s\geq k}}
\rB\{n,m\} \frac{\left(m+\frac 1\alpha\right)\Gamma(m)}{\Gamma\left(m+1+\frac 2\alpha\right)} \rA\left[m+\frac 2\alpha, k+\frac 2\alpha\right]\label{eq:beta_f_vect1}\\
&=
\frac{2\cdot n!}{k!} \left(\frac{\Gamma(\frac{\alpha}{2})}{2\sqrt \pi \, \Gamma(\frac{\alpha+1}{2})}\right)^{n-k}
\sum_{\substack{q=0,2,4,\ldots\\ d - q \geq k}} \rB\left\{n-\frac{q}{\alpha}, d - q - \frac q \alpha\right\}\rA\left[d - q - \frac q\alpha, k-\frac {q}{\alpha}\right]. \label{eq:beta_f_vect2}
\end{align}
Under the same conditions as above, the following complementary formulae hold:
\begin{align}
\lefteqn{\binom nk - \E f_{k-1} (P_{n,d}^\beta)}\notag\\
&=
\frac{2\cdot n!}{k!} \left(\frac{\Gamma(\frac{\alpha}{2})}{2\sqrt \pi \, \Gamma(\frac{\alpha+1}{2})}\right)^{n-k}
\sum_{\substack{s=2,4,6,\ldots \\ m:=d+s\leq n}}
\rB\{n,m\} \frac{\left(m+\frac 1\alpha\right)\Gamma(m)}{\Gamma\left(m+1+\frac 2\alpha\right)} \rA\left[m+\frac 2\alpha, k+\frac 2\alpha\right]
\label{eq:beta_f_vect1compl}\\
&=
\frac{2\cdot n!}{k!} \left(\frac{\Gamma(\frac{\alpha}{2})}{2\sqrt \pi \, \Gamma(\frac{\alpha+1}{2})}\right)^{n-k}
\sum_{\substack{q=2,4,6,\ldots\\ d + q \leq n}} \rB\left\{n+\frac{q}{\alpha}, d + q + \frac q \alpha\right\}\rA\left[d + q + \frac q\alpha, k+\frac {q}{\alpha}\right].
\label{eq:beta_f_vect2compl}
\end{align}
\end{theorem}
\begin{remark}
If $d=2$ and $\beta=-1$, then $\alpha=0$ and the above formulae (which involve the value $2/\alpha$) do not apply. However, this case is not interesting because $P_{n,2}^{-1}$ is a polygon with $n$ vertices. If $d=2$ and $\beta>-1$, then the formulae hold true without changes in the proof. 
\end{remark}
Formula~\eqref{eq:beta_f_vect1} has been obtained in~\cite[Theorem~7.1]{kabluchko_formula}. The equivalence between~\eqref{eq:beta_f_vect1} and~\eqref{eq:beta_f_vect2}, as well as the equivalence between~\eqref{eq:beta_f_vect1compl} and~\eqref{eq:beta_f_vect2compl} will be established in Theorem~\ref{theo:identity_main}. Finally, the equivalence between~\eqref{eq:beta_f_vect1} and~\eqref{eq:beta_f_vect1compl} is a consequence of the following proposition which we proved in Proposition~5.9 and Remark~5.10 of~\cite{kabluchko_angles}. %in~\cite[Proposition~5.11]{kabluchko_angles}.
\begin{proposition}\label{prop:relations_A_B}
Let $\alpha>0$.  If $n,k\in \C$ are complex variables such that $n-k\in \N$ is arbitrary but fixed (observe that $k=n$ is not allowed), then the following equalities of meromorphic functions in $n$ (or $k$) hold:
%For every $n\in\N$ and $k\in \{1,\ldots,n-1\}$ (but not for $k=n$) we have
\begin{align}
%\sum_{\substack{m\in \{k,\ldots,n\}\\ m \text{ is even}}}
\sum_{\substack{m\in \{k,\ldots,n\}\\ n-m \text{ is even}}}
\rB\{n,m\} \frac{\left(m + \frac 1\alpha\right)\Gamma(m)}{\Gamma\left(m+1+\frac 2\alpha\right)} \rA\left[m + \frac 2\alpha, k + \frac 2\alpha\right]
&=
\frac{1}{2 (n-k)!}\left(\frac{2\sqrt \pi \Gamma\left(\frac{\alpha+1}{2}\right)}{\Gamma\left(\frac{\alpha}{2}\right)}\right)^{n-k},
\label{eq:full_sum_1}\\
\sum_{\substack{m\in \{k,\ldots,n\}\\ n-m \text{ is odd}}}
\rB\{n,m\} \frac{\left(m + \frac 1\alpha\right)\Gamma(m)}{\Gamma\left(m+1+\frac 2\alpha\right)} \rA\left[m + \frac 2\alpha, k + \frac 2\alpha\right]
&=
\frac{1}{2 (n-k)!}\left(\frac{2\sqrt \pi \Gamma\left(\frac{\alpha+1}{2}\right)}{\Gamma\left(\frac{\alpha}{2}\right)}\right)^{n-k}.
 \label{eq:full_sum_2}
\end{align}
\end{proposition}

%\section{Proofs: Face numbers of  beta polytopes}
\subsection{Main relation for beta polytopes}
The quantities $\rB$ and $\rA$ satisfy the following relations, see Equations~(5.17) and~(5.18) of~\cite{kabluchko_formula}:
\begin{align}
&\rB\{\nu,\kappa+2\} + \rB\{\nu,\kappa\}
=
\frac{\left(\kappa - \frac 1\alpha\right) \Gamma\left(\kappa - \frac 2\alpha\right)} {\Gamma(\kappa+1)} \rB \left\{\nu - \frac 2\alpha, \kappa - \frac 2 \alpha\right\},\label{eq:b_relation}\\
&\rA[\nu-2,\kappa] + \rA[\nu,\kappa]
=
\frac{\left(\nu+\frac 1 \alpha\right)\Gamma(\nu)}{\Gamma\left(\nu+\frac 2 \alpha + 1\right)} \rA \left[\nu + \frac 2\alpha, \kappa + \frac 2 \alpha\right]. \label{eq:a_0_relation}
\end{align}
These identities can be viewed as equalities between meromorphic functions, where $\nu,\kappa\in \C$ are complex variables such that $\nu-\kappa\in \N_0$ is fixed.
%and $\Re \kappa > \frac 1\alpha$ (for~\eqref{eq:b_relation}), respectively, $\Re \kappa >0$ (for~\eqref{eq:a_0_relation}), but 
The next theorem is the main result of this section.

\begin{theorem}\label{theo:identity_main}
Fix $\alpha>0$.
Let $n,d,k\in \C$ be such that $n-d\in \N_0$ and $d-k\in \N_0$ are arbitrary but fixed. Then, the following equalities of meromorphic functions in the complex variable $k$ (or $n$) hold:
\begin{multline}\label{eq:main_1}
\sum_{\substack{s=0,2,4,\ldots \\ m:=d-s\geq k}}
\rB\{n,m\} \frac{\left(m+\frac 1\alpha\right)\Gamma(m)}{\Gamma\left(m+1+\frac 2\alpha\right)} \rA\left[m+\frac 2\alpha, k+\frac 2\alpha\right]
\\=
\sum_{\substack{q=0,2,4,\ldots\\ d - q \geq k}} \rB\left\{n-\frac{q}{\alpha}, d - q - \frac q \alpha\right\}\rA\left[d - q - \frac q\alpha, k-\frac {q}{\alpha}\right],
\end{multline}
\begin{multline}\label{eq:main_2}
\sum_{\substack{s=2,4,6,\ldots \\ m:=d+s\leq n}}
\rB\{n,m\} \frac{\left(m+\frac 1\alpha\right)\Gamma(m)}{\Gamma\left(m+1+\frac 2\alpha\right)} \rA\left[m+\frac 2\alpha, k+\frac 2\alpha\right]
\\=
\sum_{\substack{q=2,4,6,\ldots\\ d + q \leq n}} \rB\left\{n+\frac{q}{\alpha}, d  + q  +\frac q \alpha\right\}\rA\left[d + q + \frac q\alpha, k + \frac {q}{\alpha}\right].
\end{multline}
\end{theorem}

\begin{proof}
Let us prove~\eqref{eq:main_1}. We are interested in the following sum:
$$
S_d(n,k):=\sum_{\substack{s=0,2,4,\ldots \\ m:=d-s\geq k}}
\rB\{n,m\} \frac{\left(m+\frac 1\alpha\right)\Gamma(m)}{\Gamma\left(m+1+\frac 2\alpha\right)} \rA\left[m+\frac 2\alpha, k+\frac 2\alpha\right].
$$
Applying~\eqref{eq:a_0_relation}, regrouping the terms and applying~\eqref{eq:b_relation} yields
\begin{align*}
S_d(n,k)
&=
\sum_{\substack{s=0,2,4,\ldots \\ m:=d-s\geq k}}
\rB\{n,m\} \left(\rA[m-2,k] + \rA[m,k]\right)\\
&=
\rB\{n,d\} \rA[d,k]
+
\sum_{\substack{s=0,2,4,\ldots \\ m:=d-s\geq k+2}}
\rB\{n,m\} \rA[m-2,k ]
+
\sum_{\substack{s=2,4,\ldots \\ m:=d-s\geq k}}
\rB\{n,m\} \rA[m,k]\\
&=
\rB\{n,d\} \rA[d,k]
+
\sum_{\substack{s=0,2,4,\ldots \\ m:=d-s\geq k+2}}
\rB\{n,m\} \rA[m-2,k ]
+
\sum_{\substack{s=0,2,4,\ldots \\ m:=d-s\geq k+2}}
\rB\{n,m-2\} \rA[m-2,k]\\
&=
\rB\{n,d\} \rA[d,k]
+
\sum_{\substack{s=0,2,4,\ldots \\ m:=d-s\geq k+2}}
\left(\rB\{n,m\} + \rB\{n,m-2\}\right) \rA[m-2,k]\\
&=
\rB\{n,d\} \rA[d,k]
+
\sum_{\substack{s=0,2,4,\ldots \\ m:=d-s\geq k+2}}
\rB \left\{n - \frac 2\alpha, m - 2 - \frac 2 \alpha\right\}
\frac{\left(m - 2 - \frac 1\alpha\right) \Gamma\left(m - 2 - \frac 2\alpha\right)} {\Gamma(m-1)}
\rA[m-2,k]\\
&=
\rB\{n,d\} \rA[d,k]
+
\sum_{\substack{s=0,2,4,\ldots \\ m':=d - 2 - \frac 2 \alpha - s \geq k - \frac 2\alpha }}
\rB \left\{n - \frac 2\alpha, m'\right\}
\frac{\left(m' +  \frac 1\alpha\right) \Gamma\left(m'\right)} {\Gamma(m' + 1 + \frac 2\alpha)}
\rA\left[m' + \frac 2\alpha,k\right]\\
&=
\rB\{n,d\} \rA[d,k]
+
S_{d - 2 - \frac 2 \alpha}\left(n-\frac 2\alpha, k-\frac 2\alpha\right).
\end{align*}
Iterating this identity $r$ times, we arrive at
$$
S_d(n,k) = \sum_{\substack{q=0,2,4,\ldots, 2r-2\\ d - q \geq k}} \rB\left\{n - \frac{q}{\alpha}, d - q - \frac q \alpha\right\}\rA\left[d - q - \frac q\alpha, k - \frac {q}{\alpha}\right] + S_{d-2r - \frac {2r}\alpha} \left(n - \frac {2r}\alpha, d - \frac {2r}\alpha \right).
$$
For sufficiently large $r$, the last sum becomes empty and the proof of~\eqref{eq:main_1} is complete. The proof of~\eqref{eq:main_2} is very similar, but we apply the recurrence relations in the reverse order.
\end{proof}

\begin{corollary}\label{cor:main}
Let $\alpha>0$.
If $n,k\in \C$ are complex variables such that $n-k\in \N$ is arbitrary but fixed (observe that $k=n$ is not allowed), then the following equalities of meromorphic functions in $n$ (or $k$) hold:
\begin{align*}
&\sum_{\substack{q=0,2,4,\ldots\\ n - q \geq k}} \rB\left\{n-\frac{q}{\alpha}, n - q - \frac q \alpha\right\}\rA\left[n - q - \frac q\alpha, k-\frac {q}{\alpha}\right] = \frac{1}{2 (n-k)!}\left(\frac{2\sqrt \pi \Gamma\left(\frac{\alpha+1}{2}\right)}{\Gamma\left(\frac{\alpha}{2}\right)}\right)^{n-k},\\
&\sum_{\substack{q=0,2,4,\ldots\\ n  - q \geq k}} \rB\left\{n-\frac{q}{\alpha}, n - 1- q - \frac q \alpha\right\}\rA\left[n - 1 - q - \frac q\alpha, k-\frac {q}{\alpha}\right]
=
\frac{1}{2 (n-k)!}\left(\frac{2\sqrt \pi \Gamma\left(\frac{\alpha+1}{2}\right)}{\Gamma\left(\frac{\alpha}{2}\right)}\right)^{n-k}.
\end{align*}
\end{corollary}
\begin{proof}
Use Equation~\eqref{eq:main_1} of Theorem~\ref{theo:identity_main} with $d=n$ and $d=n-1$ and combine the result with Proposition~\ref{prop:relations_A_B}.
\end{proof}
Notice the following alternative way of proving~\eqref{eq:main_2}: it is a consequence of~\eqref{eq:main_1} together with Proposition~\ref{prop:relations_A_B} and Corollary~\ref{cor:main}.

\section{Face numbers of  beta' polytopes}

\subsection{Statement of the result}
To define beta' polytopes, let $X_1,X_2,\ldots$ be i.i.d.\ random points in $\R^d$ with the Lebesgue density function
\begin{equation}\label{eq:def_f_beta_prime}
\tilde{f}_{d,\beta}(x)= \frac{ \Gamma\left( \beta \right) }{\pi^{ \frac{d}{2} } \Gamma\left( \beta - \frac{d}{2} \right) } \left( 1+\left\| x \right\|^2 \right)^{-\beta},
\qquad x\in \R^d,
\end{equation}
where $\beta>d/2$ is a real parameter. % satisfying $\beta> d/2$ to ensure integrability.
This $d$-dimensional \textit{beta' distribution} appeared in the works of Miles~\cite{miles} and Ruben and Miles~\cite{ruben_miles}.  The convex hull $\tilde P_{n,d}^\beta :=[X_1,\ldots,X_n]$ is called the \textit{beta' polytope}. Several models of stochastic geometry such as the typical Poisson-Voronoi cell (in the Euclidean space or on the sphere), the zero cell of the Poisson hyperplane tessellation (again, both Euclidean and spherical), and the random convex cones in the half-space studied in Section~\ref{sec:random_convex_cones_half_space} reduce to beta' polytopes with various parameters $\beta$ or to their limits as $n\to\infty$ called Poisson polyhedra; see~\cite{convex_hull_sphere,beta_polytopes,kabluchko_poisson_zero,kabluchko_thaele_voronoi_sphere,kabluchko_thaele_great_hypersphere}.  The expected number of $k$-dimensional faces of the beta' polytope has been determined exactly and explicitly in~\cite[Section~7.2]{kabluchko_formula} relying on~\cite{beta_polytopes}. To state the corresponding formula, we need to  introduce some notation.
For  $\alpha >0$ we consider the analytic function
\begin{equation}\label{eq:def_F_tilde}
\tilde F(x) = \int_{-\pi/2}^x (\cos y)^{\alpha-1} \dd y,
\qquad
x\in \C\backslash \left(\left(-\infty, -\frac \pi 2\right)\cup \left(+\frac \pi2,\infty\right)\right) =: \mathcal H,
\end{equation}
which is the same as $F$ defined in~\eqref{eq:def_F} with $\alpha$ replaced by $\alpha-1$.
For $\nu,\kappa\in \C$ such that $\nu-\kappa\in \N_0$ we define
\begin{align}
\mtB\{\nu, \kappa\}
&:=
\frac{\alpha^{\nu-\kappa}}{\Gamma(\kappa)(\nu-\kappa)!}\int_{-\pi/2}^{+\pi/2} (\cos x)^{\alpha \kappa-1} (\tilde F(x))^{\nu-\kappa} \dd x,
\qquad \Re \kappa > 0,\label{eq:def_B_tilde}\\
\mtA[\nu,\kappa]
&:=
\frac{\alpha^{\nu - \kappa + 1}}{(\nu - \kappa)!} \cdot \frac {\Gamma(\nu+1)} {2\pi \ii} \int_{- \ii \infty}^{+\ii \infty} (\cos x)^{-\alpha \nu-1} (\tilde F(x))^{\nu-\kappa} \dd x. \label{eq:def_A_0_tilde1}
\end{align}
For $\nu, \kappa\in \C$ such that $\nu - \kappa \in \{-1,-2,\ldots\}$ we put $\mtA[\nu,\kappa]:= \mtB\{\nu, \kappa\} := 0$. 
%To ensure that the integral~\eqref{eq:def_A_0_tilde1} is absolutely convergent we have to impose the following additional conditions on  $\nu$ and $\kappa$:
%\begin{align}
%&\alpha \Re \nu + 1 > 0 \qquad  &&\text{if } \alpha \in (0,1],\label{eq:cond_nu_kappa1}\\
%&\alpha \Re \kappa + (\nu-\kappa) + 1 > 0  \qquad  &&\text{if } \alpha >1.\label{eq:cond_nu_kappa2}
%\end{align}
As in the beta case, it is possible to show that for fixed $\nu-\kappa\in \N_0$, the functions $\mtB\{\nu, \kappa\}$ and $\mtA[\nu,\kappa]$ are meromorphic in the complex variable $\nu$ (or $\kappa$). 
Note that~\cite{kabluchko_formula} also makes use of the notation
$\mB\{\nu,\kappa\} =  \mtB\{\nu,\kappa\}\Gamma(\kappa)$ and  $\mA[\nu,\kappa] = \mtA[\nu,\kappa]/\Gamma(\nu+1)$. 
We are now able to state our formulae for the expected face numbers of beta' polytopes.

\begin{theorem}\label{theo:beta_poly_f_vector_tilde}
Let $d\in\N$ and $n\in\N$ be such that $n\geq d+1$. Let also $\beta > \frac d2$ and put $\alpha:= 2\beta - d > 0$. Then, for all $k\in \{1,\ldots,d\}$ such that $\alpha k>1$ we have
\begin{align}
\lefteqn{\E f_{k-1} (\tilde P_{n,d}^\beta)}\notag\\
&=
\frac{2\cdot n!}{k!} \left(\frac{\Gamma(\frac{\alpha+1}{2})}{2\sqrt \pi \, \Gamma(\frac{\alpha+2}{2})}\right)^{n-k}
\sum_{\substack{s=0,2,4,\ldots \\ m:=d-s\geq k}}
\mtB\{n,m\} \frac{\left(m - \frac 1\alpha\right)\Gamma(m)}{\Gamma\left(m+1-\frac 2\alpha\right)} \mtA\left[m - \frac 2\alpha, k - \frac 2\alpha\right]\label{eq:beta_prime_f_vect1}\\
&=
\frac{2\cdot n!}{k!} \left(\frac{\Gamma(\frac{\alpha+1}{2})}{2\sqrt \pi \, \Gamma(\frac{\alpha+2}{2})}\right)^{n-k}
\sum_{\substack{q=0,2,4,\ldots\\ d - q \geq k}} \mtB\left\{n+\frac{q}{\alpha}, d - q + \frac q \alpha\right\}\mtA\left[d - q + \frac q\alpha, k+\frac {q}{\alpha}\right].\label{eq:beta_prime_f_vect2}
\end{align}
Also, under the same conditions as above, the following complementary identities hold:
\begin{align}
\lefteqn{\binom nk - \E f_{k-1} (\tilde P_{n,d}^\beta)} \notag\\
&=
\frac{2\cdot n!}{k!} \left(\frac{\Gamma(\frac{\alpha+1}{2})}{2\sqrt \pi \, \Gamma(\frac{\alpha+2}{2})}\right)^{n-k}
\sum_{\substack{s=2,4,6,\ldots \\ m:=d+s\leq n}}
\mtB\{n,m\} \frac{\left(m - \frac 1\alpha\right)\Gamma(m)}{\Gamma\left(m+1-\frac 2\alpha\right)} \mtA\left[m - \frac 2\alpha, k - \frac 2\alpha\right]\label{eq:beta_prime_f_vect1_compl}\\
&=
\frac{2\cdot n!}{k!} \left(\frac{\Gamma(\frac{\alpha+1}{2})}{2\sqrt \pi \, \Gamma(\frac{\alpha+2}{2})}\right)^{n-k}
\sum_{\substack{q=2,4,6,\ldots\\ d + q \leq n}} \mtB\left\{n-\frac{q}{\alpha}, d + q - \frac q \alpha\right\}\mtA\left[d + q - \frac q\alpha, k- \frac {q}{\alpha}\right].\label{eq:beta_prime_f_vect2_compl}
\end{align}
\end{theorem}

Formula~\eqref{eq:beta_prime_f_vect1} has been derived in~\cite[Theorem~7.3]{kabluchko_formula}. We will prove the equivalences between~\eqref{eq:beta_prime_f_vect1} and~\eqref{eq:beta_prime_f_vect2} as well as between~\eqref{eq:beta_prime_f_vect1_compl} and~\eqref{eq:beta_prime_f_vect2_compl} in Theorem~\ref{theo:identity_main_tilde}. The equivalence between~\eqref{eq:beta_prime_f_vect1} and~\eqref{eq:beta_prime_f_vect1_compl} follows from the next proposition which we proved in Proposition~6.7 and Remark~6.8 of~\cite{kabluchko_angles}. %in~\cite[Proposition~6.9]{kabluchko_angles}.

\begin{proposition}\label{prop:relations_A_B_tilde}
Let $\alpha>0$.  If $n,k\in \C$ are complex variables such that $n-k\in \N$ is arbitrary but fixed (observe that $k=n$ is not allowed), then the following equalities of meromorphic functions in $n$ (or $k$) hold:
%For every $n\in\N$ and $k\in \{1,\ldots,n-1\}$ (but not for $k=n$)  such that $\alpha k>1$ we have
\begin{align}
\sum_{\substack{m\in \{k,\ldots,n\}\\ n-m \text{ is even}}}
\mtB\{n,m\} \frac{\left(m - \frac 1\alpha\right)\Gamma(m)}{\Gamma\left(m+1-\frac 2\alpha\right)} \mtA\left[m - \frac 2\alpha, k - \frac 2\alpha\right]
&=
\frac{1}{2 (n-k)!}\left(\frac{2\sqrt \pi \Gamma\left(\frac{\alpha+2}{2}\right)}{\Gamma\left(\frac{\alpha+1}{2}\right)}\right)^{n-k},
\label{eq:full_sum_tilde_1}\\
\sum_{\substack{m\in \{k,\ldots,n\}\\ n-m \text{ is odd}}}
\mtB\{n,m\} \frac{\left(m - \frac 1\alpha\right)\Gamma(m)}{\Gamma\left(m+1-\frac 2\alpha\right)} \mtA\left[m - \frac 2\alpha, k - \frac 2\alpha\right]
&=
\frac{1}{2 (n-k)!}\left(\frac{2\sqrt \pi \Gamma\left(\frac{\alpha+2}{2}\right)}{\Gamma\left(\frac{\alpha+1}{2}\right)}\right)^{n-k}.
 \label{eq:full_sum_tilde_2}
\end{align}
\end{proposition}

\subsection{Main relation for beta' polytopes}

The quantities $\mtA$ and $\mtB$ satisfy the following recurrence relations which were established in Equations~(6.15) and~(6.16) of~\cite{kabluchko_formula}:
\begin{align}
&\mtB\{\nu,\kappa\} - \mtB\{\nu,\kappa+2\}
=
\frac{\left(\kappa  +  \frac 1\alpha\right) \Gamma\left(\kappa  + \frac 2\alpha\right)} {\Gamma(\kappa+1)} \mtB \left\{\nu + \frac 2\alpha, \kappa + \frac 2 \alpha\right\},\label{eq:b_relation_tilde}\\
&
\mtA[\nu,\kappa] - \mtA[\nu-2,\kappa]
=
\frac{\left(\nu - \frac 1 \alpha\right)\Gamma(\nu)}{\Gamma\left(\nu - \frac 2 \alpha + 1\right)} \mtA \left[\nu  -  \frac 2\alpha, \kappa -  \frac 2 \alpha\right]. \label{eq:a_0_relation_tilde}
\end{align}
These identities are equalities between meromorphic functions, where $\nu,\kappa\in \C$ are complex variables such that $\nu-\kappa\in \N_0$ is fixed. We shall use these recurrence relations to derive the following 
\begin{theorem}\label{theo:identity_main_tilde}
Fix $\alpha>0$.
Let $n,d,k\in \C$ be such that $n-d\in \N_0$ and $d-k\in \N_0$ are arbitrary but fixed. Then, the following equalities of meromorphic functions in the complex variable $k$ (or $n$) hold:
\begin{multline}\label{eq:main_1tilde}
\sum_{\substack{s=0,2,4,\ldots \\ m:=d-s\geq k}}
\mtB\{n,m\} \frac{\left(m - \frac 1\alpha\right)\Gamma(m)}{\Gamma\left(m+1-\frac 2\alpha\right)} \mtA\left[m - \frac 2\alpha, k - \frac 2\alpha\right]
\\=
\sum_{\substack{q=0,2,4,\ldots\\ d - q \geq k}} \mtB\left\{n+\frac{q}{\alpha}, d - q + \frac q \alpha\right\}\mtA\left[d - q + \frac q\alpha, k+\frac {q}{\alpha}\right],
\end{multline}
\begin{multline}\label{eq:main_2tilde}
\sum_{\substack{s=2,4,6,\ldots \\ m:=d+s\leq n}}
\mtB\{n,m\} \frac{\left(m - \frac 1\alpha\right)\Gamma(m)}{\Gamma\left(m+1-\frac 2\alpha\right)} \mtA\left[m - \frac 2\alpha, k - \frac 2\alpha\right]
\\=
\sum_{\substack{q=2,4,6,\ldots\\ d + q \leq n}} \mtB\left\{n-\frac{q}{\alpha}, d + q - \frac q \alpha\right\}\mtA\left[d + q - \frac q\alpha, k-\frac {q}{\alpha}\right].
\end{multline}
\end{theorem}
\begin{proof}
We prove only~\eqref{eq:main_1tilde} since the proof of~\eqref{eq:main_2tilde} is similar. We are interested in the following sum:
\begin{align*}
\tilde S_d(n,k) := \sum_{\substack{s=0,2,4,\ldots \\ m:=d-s\geq k}}
\mtB\{n,m\} \frac{\left(m - \frac 1\alpha\right)\Gamma(m)}{\Gamma\left(m+1-\frac 2\alpha\right)} \mtA\left[m - \frac 2\alpha, k - \frac 2\alpha\right].
\end{align*}
Applying the usual procedure involving the use of~\eqref{eq:b_relation_tilde} and~\eqref{eq:a_0_relation_tilde}, we get
\begin{align*}
\tilde S_d(n,k)
&=
\sum_{\substack{s=0,2,4,\ldots \\ m:=d-s\geq k}} \mtB\{n,m\}\left( \mtA\left[m, k\right] - \mtA[m-2 , k]\right)\\
&=
\mtB\{n,d\} \mtA[d,k] + \sum_{\substack{s=2,4,6,\ldots \\ m:=d-s\geq k}} \mtB\{n,m\} \mtA\left[m, k\right]
-
\sum_{\substack{s=0,2,4,\ldots \\ m:=d-s\geq k+2}} \mtB\{n,m\} \mtA[m-2,k]\\
&=
\mtB\{n,d\} \mtA[d,k] + \sum_{\substack{s=0,2,4,\ldots \\ m:=d-s\geq k+2}} \mtB\{n,m-2\} \mtA\left[m-2, k\right]
-
\sum_{\substack{s=0,2,4,\ldots \\ m:=d-s\geq k+2}} \mtB\{n,m\} \mtA[m-2,k]\\
&=
\mtB\{n,d\} \mtA[d,k]
+
\sum_{\substack{s=0,2,4,\ldots \\ m:=d-s\geq k+2}} (\mtB\{n,m-2\} - \mtB\{n,m\})  \mtA\left[m-2, k\right]\\
&=
\mtB\{n,d\} \mtA[d,k] + \sum_{\substack{s=0,2,4,\ldots \\ m:=d-s\geq k+2}}   \mtB \left\{n + \frac 2\alpha, m-2 + \frac 2 \alpha\right\} \frac{\left(m-2  +  \frac 1\alpha\right) \Gamma\left(m-2  + \frac 2\alpha\right)} {\Gamma(m-1)}  \mtA\left[m-2, k\right]\\
&=
\mtB\{n,d\} \mtA[d,k] + \sum_{\substack{s=0,2,4,\ldots \\ m':=d - 2 + \frac 2\alpha - s \geq k + \frac 2\alpha}}
\mtB \left\{n + \frac 2\alpha, m'\right\} \frac{\left(m'  -  \frac 1\alpha\right) \Gamma\left(m'\right)} {\Gamma(m'+1- \frac 2\alpha)}  \mtA\left[m'-\frac 2\alpha, k\right]\\
&=
\mtB\{n,d\} \mtA[d,k] + \tilde S_{d-2+\frac 2\alpha} \left(n+\frac 2\alpha, k+\frac 2\alpha \right).
\end{align*}
Applying this identity $r$ times, we arrive at
$$
\tilde S_d(n,k) = \sum_{\substack{q=0,2,4,\ldots, 2r-2\\ d - q \geq k}} \mtB\left\{n+\frac{q}{\alpha}, d - q + \frac q \alpha\right\}\mtA\left[d - q + \frac q\alpha, k+\frac {q}{\alpha}\right] + \tilde S_{d-2r+\frac {2r}\alpha} \left(n+\frac {2r}\alpha, k+\frac {2r}\alpha \right).
$$
For sufficiently large $r$, the sum becomes empty and the proof of~\eqref{eq:main_1tilde} is complete.
\end{proof}

\begin{corollary}\label{cor:main_tilde}
Let $\alpha>0$.
If $n,k\in \C$ are complex variables such that $n-k\in \N$ is arbitrary but fixed (observe that $k=n$ is not allowed), then the following equalities of meromorphic functions in $n$ (or $k$) hold:
\begin{align*}
&\sum_{\substack{q=0,2,4,\ldots\\ n - q \geq k}} \mtB\left\{n+\frac{q}{\alpha}, n - q + \frac q \alpha\right\}\mtA\left[n - q + \frac q\alpha, k+\frac {q}{\alpha}\right]
=\frac{1}{2 (n-k)!}\left(\frac{2\sqrt \pi \Gamma\left(\frac{\alpha+2}{2}\right)}{\Gamma\left(\frac{\alpha+1}{2}\right)}\right)^{n-k},\\
&\sum_{\substack{q=0,2,4,\ldots\\ n - 1 - q \geq k}} \mtB\left\{n+\frac{q}{\alpha}, n - 1-  q + \frac q \alpha\right\}\mtA\left[n - 1 - q + \frac q\alpha, k+\frac {q}{\alpha}\right]
=\frac{1}{2 (n-k)!}\left(\frac{2\sqrt \pi \Gamma\left(\frac{\alpha+2}{2}\right)}{\Gamma\left(\frac{\alpha+1}{2}\right)}\right)^{n-k}.
\end{align*}
\end{corollary}
\begin{proof}
Use Theorem~\ref{theo:identity_main_tilde} with $d=n$ and $d=n-1$ and combine the result with Proposition~\ref{prop:relations_A_B_tilde}.
\end{proof}

Observe that~\eqref{eq:main_2tilde} can also be proved in the following way: it is a consequence of~\eqref{eq:main_1tilde} together with Proposition~\ref{prop:relations_A_B_tilde} and Corollary~\ref{cor:main_tilde}.

\section{Identities for Stirling numbers}
The proofs of the above results are purely algebraic manipulations using the recurrence relations satisfied by the quantities $\rA$, $\rB$, $\mtA$, $\mtB$. These relations bear some similarity to the following well-known~\cite[\S 6.1]{graham_knuth_patashnik_book} relations satisfied by the Stirling numbers of both kinds, denoted by $\stirling{n}{k}$ and $\stirlingb{n}{k}$:
%These properties bear some similarity to the well-known properties of the Stirling numbers $\stirling{n}{k}$ and $\stirlingb{n}{k}$:
\begin{equation}\label{eq:stirling_relations}
\stirling{n+1}{k} - \stirling{n}{k-1} = n\stirling{n}{k},
\qquad
\stirlingb{n+1}{k} - \stirlingb{n}{k-1} = k\stirlingb{n}{k}.
\end{equation}
One may therefore ask whether similar manipulations also yields some identities for the Stirling numbers. This is indeed the case, as the next theorem shows.

\begin{theorem}\label{theo:stirling_identities}
For all $n\in \N$, $d\in \{0,\ldots,n\}$ and $k\in \{0,\ldots,d\}$ the following identities hold true:
$$
\sum_{s=0}^k \stirlingb{n-s}{d-s} (d-s) \stirling{d-s}{k-s}
=
\sum_{s=0}^k (-1)^s \stirlingb{n-s}{d} \stirling{d+1}{k-s}
=
\sum_{s=0}^{d-k} (-1)^s \stirlingb{n+1}{d-s} \stirling{d-s}{k}.
$$
\end{theorem}

\begin{proof}
We are going to transform the following sum in two different ways:
$$
S_d(n,k) := \sum_{s=0}^{k-1} \stirlingb{n-s}{d-s} (d-s) \stirling{d-s}{k-s}.
$$
Using the relation for $\stirling{n}{k}$, regrouping the terms and then using the relation for $\stirlingb{n}{k}$ yields
\begin{align*}
S_{d}(n,k)
&=
\sum_{s=0}^{k-1} \stirlingb{n-s}{d-s} \left( \stirling{d-s+1}{k-s} - \stirling{d-s}{k-s-1}\right)\\
&=
\stirlingb{n}{d} \stirling{d+1}{k} +  \sum_{s=1}^{k-1} \stirlingb{n-s}{d-s} \stirling{d-s+1}{k-s} - \sum_{s=0}^{k-2} \stirlingb{n-s}{d-s}\stirling{d-s}{k-s-1}\\
&=
\stirlingb{n}{d} \stirling{d+1}{k} +  \sum_{s=0}^{k-2} \stirlingb{n-s-1}{d-s-1} \stirling{d-s}{k-s-1} - \sum_{s=0}^{k-2} \stirlingb{n-s}{d-s}\stirling{d-s}{k-s-1}\\
&=
\stirlingb{n}{d} \stirling{d+1}{k} +  \sum_{s=0}^{k-2} \left(\stirlingb{n-s-1}{d-s-1} - \stirlingb{n-s}{d-s}\right) \stirling{d-s}{k-s-1}\\
&=
\stirlingb{n}{d} \stirling{d+1}{k} -  \sum_{s=0}^{k-2} \stirlingb{n-s-1}{d-s} (d-s) \stirling{d-s}{k-s-1}\\
&=
\stirlingb{n}{d} \stirling{d+1}{k} - S_{d}(n-1,k-1).
\end{align*}
Repeating the above $r$ times  gives
\begin{align*}
S_{d}(n,k)
&=
\stirlingb{n}{d} \stirling{d+1}{k} - \stirlingb{n-1}{d} \stirling{d+1}{k-1} + \ldots +(-1)^{r-1} \stirlingb{n-r+1}{d} \stirling{d+1}{k-r+1} \\ &+(-1)^{r} S_{d}(n-r,k-r).
\end{align*}
For sufficiently large $r$, the last term on the right-hand side becomes an empty sum, and we arrive at the first identity in Theorem~\ref{theo:stirling_identities}.

To prove the second identity, we transform $S_d(n,k)$ by first using the relation for $\stirlingb{n}{k}$, regrouping the terms and then using the relation for $\stirling{n}{k}$ as follows:
\begin{align*}
S_{d}(n,k)
&=
\sum_{s=0}^{k-1} \stirlingb{n-s}{d-s} (d-s) \stirling{d-s}{k-s}\\
&=
\sum_{s=0}^{k-1} \left(\stirlingb{n-s+1}{d-s} - \stirlingb{n-s}{d-s-1}\right) \stirling{d-s}{k-s}\\
&=
\stirlingb{n+1}{d} \stirling{d}{k} + \sum_{s=1}^{k-1} \stirlingb{n-s+1}{d-s}\stirling{d-s}{k-s} - \sum_{s=0}^{k-1} \stirlingb{n-s}{d-s-1}\stirling{d-s}{k-s}\\
&=
\stirlingb{n+1}{d} \stirling{d}{k} + \sum_{s=0}^{k-2} \stirlingb{n-s}{d-s-1}\stirling{d-s-1}{k-s-1} - \sum_{s=0}^{k-1} \stirlingb{n-s}{d-s-1}\stirling{d-s}{k-s}\\
&=
\stirlingb{n+1}{d} \stirling{d}{k} + \sum_{s=0}^{k-1} \stirlingb{n-s}{d-s-1} \left(\stirling{d-s-1}{k-s-1} - \stirling{d-s}{k-s}\right)\\
&=
\stirlingb{n+1}{d} \stirling{d}{k} - \sum_{s=0}^{k-1} \stirlingb{n-s}{d-s-1} (d-s-1) \stirling{d-s-1}{k-s} \\
&=
\stirlingb{n+1}{d} \stirling{d}{k} - S_{d-1}(n,k).
\end{align*}
Repeating the above $r$ times, we get
$$
S_{d}(n,k)
=
\stirlingb{n+1}{d} \stirling{d}{k} - \stirlingb{n+1}{d-1} \stirling{d-1}{k} + \ldots + (-1)^{r-1} \stirlingb{n+1}{d-r+1} \stirling{d-r+1}{k} +(-1)^r S_{d-r}(n,k).
$$
For sufficiently large $r$, the last term on the right-hand side becomes an empty sum and the proof of the second  identity in Theorem~\ref{theo:stirling_identities} is complete.
\end{proof}

\section*{Acknowledgement}
Supported by the German Research Foundation under Germany's Excellence Strategy  EXC 2044 -- 390685587, Mathematics M\"unster: Dynamics - Geometry - Structure.

%\addcontentsline{toc}{section}{References}
%:Referenzen

\bibliography{face_numbers_formulae_bib}
\bibliographystyle{plainnat}

\end{document}